\documentclass[a4paper]{amsart}

\usepackage{amssymb}
\usepackage{mathtools}

\newcommand{\myhyp}{\textnormal{-}\hskip0pt\relax}
\newcommand{\mydhyp}{\textnormal{--}\hskip0pt\relax}

\numberwithin{equation}{section}

\usepackage[dvipsnames]{xcolor}

\newcommand{\N}{\mathbb{N}}
\newcommand{\C}{\mathbb{C}}

\newcommand{\Bcal}{\mathcal{B}}

\newcommand{\Ecal}{\mathcal{E}}
\newcommand{\Fcal}{\mathcal{F}}

\newcommand{\Hcal}{\mathcal{H}}
\newcommand{\Kcal}{\mathcal{K}}

\newcommand{\Mcal}{\mathcal{M}}

\newcommand{\Ucal}{\mathcal{U}}
\newcommand{\Vcal}{\mathcal{V}}

\newcommand{\e}{\varepsilon}

\newcommand{\indF}{\mathbf{1}}

\newcommand{\contO}[1]{C_{0}({#1})}

\newcommand{\multAlg}[1]{\Mcal({#1})}

\newcommand{\cstar}{C$^\ast$\myhyp{}algebra}
\newcommand{\cstarSub}{C$^\ast$\myhyp{}subalgebra}
\newcommand{\cstarDiag}{C$^\ast$\myhyp{}diagonal}
\newcommand{\cstarPair}{C$^\ast$\myhyp{}pair}
\newcommand{\calg}[1]{\mathrm{C}^\ast({#1})}

\newcommand{\bop}[1]{\Bcal({#1})}
\newcommand{\kop}[1]{\Kcal({#1})}

\newcommand{\Croe}[1]{C^{\ast}_{\mathrm{Roe}}(#1)}
\newcommand{\Cu}[1]{C^{\ast}_{\mathrm{u}}({#1})}
\newcommand{\Cfp}[1]{C^{\ast}_{\mathrm{fp}}({#1})}

\newcommand{\Croesub}[1]{\C_{\mathrm{Roe}}[{#1}]}
\newcommand{\Cusub}[1]{\C_{\mathrm{u}}[{#1}]}
\newcommand{\Cfpsub}[1]{\C_{\mathrm{fp}}[{#1}]}

\newcommand{\CroeCartan}[1]{\ell^{\infty}_{\mathrm{Roe}}({#1})}
\newcommand{\CfpCartan}[1]{\ell^{\infty}_{\mathrm{fp}}({#1})}

\newcommand{\coaEq}{\overset{\mathrm{ce}}{\sim}}

\newcommand{\id}[1]{\mathrm{Id}_{#1}}

\newcommand{\rAr}{\longrightarrow}

\newcommand{\prn}[1]{{(#1)}}

\DeclareMathOperator{\solim}{\mathrm{so\myhyp{}}\lim}

\DeclareMathOperator{\img}{Im}
\DeclareMathOperator{\dist}{d}
\newcommand{\diam}[1]{\mathrm{diam} ({#1})}

\newcommand{\norm}[1]{{\left\lVert #1 \right \rVert}}

\newcommand{\normreg}[1]{{\lVert #1  \rVert}}
\newcommand{\normbig}[1]{{\big\lVert #1 \big \rVert}}

\newcommand{\ball}[3]{{B_{#1} (#2,#3)}}

\newcommand{\distX}[2]{{\dist_X ( #1, #2 )}}

\newcommand{\supp}[1]{\text{supp}( #1 )}
\newcommand{\prop}[1]{\operatorname{prop} ( {#1} )}

\newcommand{\normalisers}[2]{{\mathcal{N}}_{#2}{\big({#1}\big)}}

\newcommand{\dimnuc}[1]{\dim_{\mathrm{nuc}}({#1})}

\newcommand{\ddim}[2]{\dim_{\mathrm{diag}}({#1}\subseteq{#2})}

\newcommand{\ddimn}[3]{\dim_{\mathrm{diag}}({#1}\subseteq{#2}\, ;~{#3})}

\newcommand{\asdim}[1]{\mathrm{asdim}({#1})}

\newcommand{\unMatImg}[3]{{\Ecal_{#3}^{(#1), #2}}}

\newcommand{\CoProp}[1]{CoP for ${#1}$}

\usepackage[
  colorlinks=true, 
  linkcolor=black,  
  citecolor=Blue,  
  urlcolor=blue, 
  linktocpage=true
]{hyperref}

\hypersetup{
    pdftitle    =   {Generalised diagonal dimension and applications to large-scale geometry},
    pdfauthor   =   {Christos Kitsios},
    pdfsubject  =   {Preprint}
}

    \theoremstyle{plain}
    \newtheorem{thm}{Theorem}[section]        
    \newtheorem{propo}[thm]{Proposition}        
    \newtheorem{lem}[thm]{Lemma}          
    \newtheorem{cor}[thm]{Corollary}      
            
    \newtheorem{conj}[thm]{Conjecture}

    \newtheorem*{thm*}{Theorem}           
    \newtheorem*{propo*}{Proposition}      
    \newtheorem*{lem*}{Lemma}            
    \newtheorem*{cor*}{Corollary}         
    \newtheorem*{qu*}{Question}          
    \newtheorem*{conj*}{Conjecture}       
    \newtheorem*{prob*}{Problem}       
    \newtheorem*{fact*}{Fact}
    \newtheorem*{case*}{Case}

    \newtheoremstyle{indented}
        {\topsep}  
        {\topsep}  
        { \leftskip=\parindent 
            \rightskip=\parindent  } 
        {  } 
        {\itshape}  
        {.}        
        { }         
        { }         
    
    \theoremstyle{indented}

    \numberwithin{equation}{section}
    
    \theoremstyle{definition}
    \newtheorem{defin}[thm]{Definition}

    \newtheorem{defin*}{Definition}
    \newtheorem*{notation*}{Notation} 
    \newtheorem*{convention*}{Convention}

    \theoremstyle{remark}
    \newtheorem{rmk}[thm]{Remark}            
    \newtheorem{example}[thm]{Example}
    
    \newtheorem*{rmk*}{Remark}   

    \theoremstyle{plain}
    \newtheorem{alphthm}{Theorem}

    \newtheorem{alphpropo}[alphthm]{Proposition}
    \newtheorem{alphconj}[alphthm]{Conjecture}
    
    \theoremstyle{definition}
    \newtheorem{alphdef}[alphthm]{Definition}  

\begin{document}

\title[Generalised Diagonal Dimension]{Generalised diagonal dimension and applications to large-scale geometry}
\date{November 14, 2026}

\begin{abstract}
In this paper, we introduce a generalised diagonal dimension. We explain why the generalised diagonal dimension extends the notion of diagonal dimension defined by Li, Liao, and Winter, and under which conditions these dimensions coincide. We prove permanence properties for the generalised diagonal dimension and compare it with the nuclear dimension. We investigate applications of the generalised diagonal dimension in large-scale geometry; specifically, we show that the generalised diagonal dimension of a noncommutative Cartan subalgebra in the C*-algebra generated by finite-propagation operators on a uniformly locally finite metric space is bounded above by the asymptotic dimension of the space, and we conjecture that the two are equal.

\end{abstract}

\author{Christos Kitsios}
\address{Mathematisches Institut, Georg-August-Universit\"{a}t G\"{o}ttingen, Bunsenstr. 3-5, 37073 G\"{o}ttingen, Germany.}
\email{christos.kitsios@uni-goettingen.de}

\maketitle

Large-scale geometry, also known as coarse geometry, is a framework that studies
``large\myhyp{}scale'' properties of spaces, ignoring their ``local'' structure.
Consequently, we identify spaces that have the same large-scale structure, 
even if they differ locally.
More precisely, we study spaces up to \emph{coarse equivalences}.

For two proper metric spaces $X$ and $Y$ 
we define coarse equivalences 
as follows.
A map
${f \colon X~\rAr~Y}$ is called \emph{controlled} if for each ${r>0}$
there exists ${s>0}$, such that if 
${\dist_X(x,x')\le r}$, 
then 
${\dist_Y(f(x),f(x'))\le s}$.
We say that $X$ and $Y$ are \emph{coarsely equivalent}, and denote ${{X}\coaEq{Y}}$,
if there exist controlled maps ${f \colon X \rAr Y}$ and ${g\colon Y \rAr X}$ and
${C>0}$ such that ${\dist_X({g}\circ{f}(x),x)\le C~\forall~x \in X}$ and ${\dist_Y({f}\circ{g}(y),y)\le C~\forall~y \in Y}$.
Coarse geometry is the study of geometric properties that are invariant under coarse equivalences.
For a thorough treatment of coarse geometry, we refer the reader to  the books~\cite{NowakYu2023LargeScaleGeometry, Roe1996IndexTheoryCoarseGeometryTopology, Roe2003lecturesCrsGeometry}.

Properties that are preserved under coarse equivalences are known as \emph{coarse invariants}. 
The {asymptotic dimension} is an example of a coarse invariant.
It 
was introduced by Gromov for finitely generated groups~\cite{Gromov1993AsymptoticInvariantsInfiniteGroups} and it was extended for coarse spaces in~\cite{Roe2003lecturesCrsGeometry} by Roe.  
We say that a metric space $X$ has \emph{asymptotic dimension} at most $d$ and write ${\asdim{X}\le d}$ if:
for any ${r>0}$, there exists a uniformly bounded covering $\Ucal$ of $X$
with a decomposition 
\[\Ucal = \Ucal^{\prn{0}} \sqcup \hdots \sqcup \Ucal^{\prn{d}}\]
such that ${\dist(U , U') > r}$ for all distinct ${U , U' \in \Ucal^{\prn{i}}}$ and 
each ${i \in \{0,\hdots,d\}}$. 
If $X$ and $Y$ are coarsely equivalent, then it can be shown that ${\asdim{X}=\asdim{Y}}$  
and, therefore, the asymptotic dimension is a coarse invariant. 

\emph{Roe algebras} provide a link between the subjects of large\myhyp{}scale geometry and operator algebras, as they can be interpreted as a \cstar{}ic counterpart to the large\myhyp{}scale geometry of spaces.
Roe first introduced them in order to define higher indices of differential operators on Riemannian manifolds~\cite{Roe1988IndexThm1, Roe1988IndexThm2, Roe1993CoarseCohomologyIndexTheory, Roe1996IndexTheoryCoarseGeometryTopology}.
More precisely,
for a Riemannian manifold $M$ one can define a Hilbert space $L^2(M)$, and,
then
the Roe algebra of $M$, denoted by $\Croe{M}$, 
which is the \cstar{} generated by the locally compact operators of finite propagation in $\bop{L^2(M)}$. 
The K\myhyp{}theory groups of the Roe algebra were used to define higher indices of differential operators. 

Similarly, one can define the Roe algebra of a proper metric space.
To do so, fix an ample geometric module $\Hcal_X$ over X and define the Roe algebra $\Croe{X}$ of $X$ 
to be the \cstarSub{} of $\bop{\Hcal_X}$ that is
generated by
the locally compact operators of finite propagation. 
It should be stressed that different choices of ample geometric modules for $X$ give rise to isomorphic \cstar{}s.
We refer the reader to~\cite{WillettYu2020HigherIndexTheory} for an introduction to Roe algebras of proper metric spaces.

Apart from the Roe algebra, we are also interested in other related operator algebras.
The \cstar{} $\Cfp{X}$ is the \cstarSub{} of $\bop{\Hcal_X}$ generated by the operators of finite propagation and is called \emph{the \cstar{} of operators of finite propagation of $X$}. 
In the literature, finite\myhyp{}propagation operators are also known as \emph{band\myhyp{}dominated operators}~\cite{BragaVignati2023GeldandTypeDuality}.
Moreover, if $X$ is a discrete metric space, then 
we can take the (non\myhyp{}ample) geometric module $\ell^2(X)$ over $X$
and define the \emph{uniform Roe algebra of $X$}, denoted by $\Cu{X}$, to be the
\cstarSub{} of $\bop{\ell^2(X)}$ generated by the operators of finite propagation.
Following the naming convention of~\cite{MartinezVigolo2023RoeViaMods},
we refer to the above \cstar{}s (${\Croe{X},~\Cfp{X}}$ and ${\Cu{X}}$) as \emph{Roe\myhyp{}like algebras}. 
In their work, Mart{\'i}nez and Vigolo ~\cite{MartinezVigolo2023RoeViaMods} provide a unified framework 
to deal with either of the Roe\myhyp{}like algebras via coarse geometric modules.
Roe\myhyp{}like algebras have been studied extensively as operator algebras and for their connections with coarse geometry, see~\cite{BaudierBragaFarahKhukhroVignatiWillett2022UniformRoeAlgRigid, KhukhroLiVigoloZhang2021StructureAsymptoticExpanders, Roe2003lecturesCrsGeometry, Sako2014ProperyAandONLPropery,  WhiteWillett2020CartanUniformRoe, Willett2009NotesPropertyA}.

Following the introduction of Roe algebras, it was 
observed
that coarsely equivalent metric spaces have isomorphic Roe algebras \cite{HigsonRoeYu1993CoarseMayerVietoris, Roe1993CoarseCohomologyIndexTheory}. 
One could further ask whether the converse also holds; this problem is known as \emph{C$^\ast$\myhyp{}rigidity}. 
After the foundational result of {\v{S}}pakula and Willett~\cite{SpakulaWillett2013RigidityRoeAlgebras}, 
many papers have advanced the state of the art
by proving the rigidity in increasingly general settings. 
Notable works include 
\cite{BaudierBragaFarahKhukhroVignatiWillett2022UniformRoeAlgRigid,  BragaFarahVignati2020EmbeddingsUniformRoe, BragaFarahVignati2021UniformRoeCoronas, BragaFarahVignati2022GeneralUniformRoeRigidity}
for the rigidity of uniform Roe algebras 
and~\cite{BragaChungLi2020CoarseBaumConnesandRigidity, LiSpakulaZhang2023MeasuredAsymptoticExpandersRigidity} 
for Roe algebras.
The C$^\ast$\myhyp{}rigidity problem for Roe algebras on spaces of bounded geometry  
was fully solved by Mart{\'i}nez and Vigolo in~\cite{MartinezVigolo2025RigidityBoundedGeometry}.
Vignati in \cite{Vignati2026RigidityUniformRoe} proved that two uniformly locally finite metric spaces with isomorphic uniform Roe algebras are bijectively coarsely equivalent, 
and, therefore,
solved the 
C$^\ast$\myhyp{}rigidity problem for uniform Roe algebras.

\begin{alphthm}[{C$^\ast$\myhyp{}rigidity}]
    Let $X$ and $Y$ be proper metric spaces of bounded geometry. 
    Then the following are equivalent:
    \begin{enumerate} \renewcommand{\labelenumi}{(\roman{enumi})}
        \item $X$ and $Y$ are coarsely equivalent,
        \item ${\Croe{X}\cong \Croe{Y}}$,
        \item ${\Cfp{X}\cong \Cfp{Y}}$.
    \end{enumerate}
    Moreover, if $X$ and $Y$ are uniformly locally finite, then the above are also equivalent to the following:
    \begin{enumerate}
        \item[(iv)] $\Cu{X}$ and $\Cu{Y}$ are stably isomorphic, i.e. ${\Cu{X}\otimes \Kcal \cong \Cu{Y} \otimes \Kcal}$.
    \end{enumerate}
\end{alphthm}

It should be noted that part \prn{iv} of the theorem above was already proved in \cite{BaudierBragaFarahKhukhroVignatiWillett2022UniformRoeAlgRigid}.

\smallskip

Having 
C$^\ast$\myhyp{}rigidity at hand, 
it is natural to ask whether coarse invariants, and particularly the asymptotic dimension, can be detected in Roe\myhyp{}like algebras. 
Yu in \cite{Yu1997ZeroInTheSpectrumAsDim, Yu1998NovikovAsDim} highlighted the significance of the connection between the asymptotic dimension and Roe-like algebras, and their relevance to the Novikov conjecture and the Zero\myhyp{}in\myhyp{}the\myhyp{}Spectrum conjecture.
A connection between asymptotic dimension and nuclear dimension was made by Winter and Zacharias in their work on nuclear dimension \cite{WinterZacharias2010nucDim}.
Li, Liao and Winter in~\cite{LiLiaoWinter2023diagDim} further explored this connection as an application of diagonal dimension. 
In their work, they introduced a version of nuclear dimension for diagonal \cstarSub{}s, 
called diagonal dimension.

\pagebreak

\begin{alphdef}[{Diagonal Dimension,~\cite[Definition~2.1]{LiLiaoWinter2023diagDim}}] \label{intro:def:diagDim}

        Let ${(D_A \subseteq A)}$ be a \cstarSub{} with $D_A$ abelian. 
    We say ${(D_A \subseteq A)}$ has \emph{diagonal dimension} at most $d$, written as
    \[\ddim{D_A}{A} \le d,\]
    if for any finite subset ${\Fcal \subseteq A}$ and ${\e > 0}$, 
    there exist a finite\myhyp{}dimensional  \linebreak 
    C$^{\ast}$\myhyp{}algebra  
    ${F=F^{(0)} \oplus F^{(1)} \oplus \hdots \oplus F^{(d)}}$ with a maximal abelian $\ast$\myhyp{}subalgebra \linebreak
    ${D_F=D^{(0)} \oplus D^{(1)} \oplus \hdots \oplus D^{(d)}}$ 
    and completely positive maps 
    \[ A \overset{\psi}{\longrightarrow} F \overset{\phi}{\longrightarrow} A\]
    such that
    \begin{enumerate} \renewcommand{\labelenumi}{(\arabic{enumi})}
        \item $\psi$ is contractive,
        \item $\norm{\phi \circ \psi(a) - a } < \e $ for every $a \in \Fcal$,
        \item $\phi^{\prn{i}} \coloneq \phi|_{F^{(i)}}$ is a contractive order zero map, i.e. it preserves orthogonality for each ${i \in  \{0,\hdots,d\}}$,
        \item ${\psi(D_A) \subseteq D_F}$,
        \item ${\phi^{\prn{i}} \big( \normalisers{D^{(i)}}{F^{(i)}}\big) \subseteq \normalisers{D_A}{A}}$.
        
    \end{enumerate}
\end{alphdef}

Let $X$ be a uniformly locally finite metric space. 
Note that the algebra of the multiplication operators 
${\ell^\infty\prn{X}}$ in ${\bop{\ell^2\prn{X}}}$ 
is a maximal abelian $\ast$\myhyp{}subalgebra of the uniform Roe algebra $\Cu{X}$.
Li, Liao and Winter showed that the diagonal dimension of the pair ${(\ell^\infty\prn{X} \subseteq \Cu{X})}$ is equal to the asymptotic dimension of $X$.

\begin{alphthm}[{\cite[Theorem~7.7]{LiLiaoWinter2023diagDim}}] 
    Let $X$ be a uniformly locally finite metric space. Then
    \begin{equation*}
        \ddim{\ell^\infty\prn{X}}{\Cu{X}}=\asdim{X}.
    \end{equation*}
\end{alphthm}

We are interested in computing the asymptotic dimension of a space via its Roe algebra or its \cstar{} of finite\myhyp{}propagation operators.
However, the diagonal dimension and the theory of~\cite{LiLiaoWinter2023diagDim} cannot be applied to these \cstar{}s. 
One important obstruction is that 
a \cstar{} is nuclear when its diagonal dimension is finite~\cite{LiLiaoWinter2023diagDim},
while the Roe algebra of a non\myhyp{}bounded space
and the \cstar{} generated by finite\myhyp{}propagation operators 
of an ample geometric module are never nuclear.
Another obstruction is that these Roe\myhyp{}like algebras do not have a natural candidate for an abelian $\ast$\myhyp{}subalgebra.
Instead, they have a natural noncommutative Cartan subalgebra~\cite{MartinezVigolo2023RoeViaMods}, which is a noncommutative version of Cartan subalgebras introduced by Exel~\cite{Exel2011noncommCartan}
and further extended by Kwa\'{s}niewski and Meyer \cite{KwasniewskiMeyer2020NoncommutativeCartan}.

Motivated by the problem of detecting the asymptotic dimension in Roe\myhyp{}like algebras,
we introduced a generalisation of the diagonal dimension to overcome the aforementioned obstructions.

\pagebreak

\begin{alphdef}[Generalised diagonal dimension] \label{intro:def:genDiagDim}
    Fix a \cstar{} $B$. 
    Let \linebreak
    ${(D_A\subseteq{A})}$ be a nondegenerate \cstarSub{}. 
    We say ${(D_A \subseteq A)}$ has \emph{generalised diagonal dimension with respect to $B$} at most $d$, written as
    \[\ddimn{D_A}{A}{B} \le d,\]
    if for any finite subset 
    ${\Fcal \subseteq {A}}$ and ${\e > 0}$, there exist a finite\myhyp{}
    \linebreak 
    dimensional
    \cstar{} 
    ${F=F^{(0)}\oplus F^{(1)} \oplus \hdots \oplus F^{(d)}}$ 
    with canonical diagonal \linebreak
    ${D_F = D^{(0)} \oplus D^{(1)} \oplus \hdots \oplus D^{(d)}}$ 
    and completely positive maps 
    \[ {A \overset{\psi}{\longrightarrow} F \otimes B \overset{\phi}{\longrightarrow} A}\]
    such that
    {
    \begin{enumerate} \renewcommand{\labelenumi}{(\arabic{enumi})}
        \item $\psi$ is contractive,
        \item ${\normreg{\phi \circ \psi(a) - a } < \e }$ for every $a \in \Fcal$,
        \item ${\phi^{(i)} \coloneq \phi|_{F^{(i)} \otimes B}}$ is a contractive order zero map for each ${i \in \{0,\hdots,d\}}$,
        \item ${\psi(D_A) \subseteq D_F \otimes B}$,

        \item ${\phi^{(i)} \big( {D^{(i)} \otimes B} \big) \subseteq \normalisers{D_A}{A}}$
        and
        ${\phi^{(i)} ( v \otimes B )  \subseteq \normalisers{D_A}{A}}$
        for each matrix unit ${v \in F^\prn{i}}$ with respect to ${D^{\prn{i}}}$,
        
        \item for each ${i\in \{0,\hdots, d\}}$, if
        $\pi^{(i)}$ is a supporting ${\ast}$\myhyp{}homomorphism for the order zero map $\phi^{(i)}$ and ${\{u_\lambda\}_\lambda}$ is an approximate unit of $B$, then
        \[{\solim_\lambda \pi^{(i)}(v \otimes u_\lambda) \in (D_A)'},\]
        for each minimal projection ${v \in D^{\prn{i}}}$. \label{eq:condition6GenDiagDim}

    \end{enumerate}
    }
\end{alphdef}

The key idea behind our generalisation of the diagonal dimension 
is to ``enlarge the coefficients'' in the completely positive approximations:
we replace the 
finite\myhyp{}dimensional \cstar{}s with 
finite\myhyp{}dimensional \cstar{}s \emph{over a given \cstar{}},
which is the \cstar{} $B$ in the definition above.
This was inspired by the analogous comparison between the uniform Roe algebra and the \cstar{} of finite\myhyp{}propagation operators on a discrete metric space,
and it allows us to ``escape'' the nuclearity requirement of \cstar{}s with finite diagonal dimension.
It is worthwhile noting that, 
while most of the conditions in Definition~\ref{intro:def:genDiagDim} 
are clear analogues of the conditions in Definition~\ref{intro:def:diagDim}, 
Condition~\eqref{eq:condition6GenDiagDim} is a completely new requirement, which turns out to be important when dealing with infinite dimensional coefficients.
The reason why this condition is useful in our setting is that
when $A$ is a Roe\myhyp{}like algebra on a discrete metric space and $B$ is a unital \cstar{},
Condition~\eqref{eq:condition6GenDiagDim} 
should imply that
if ${u, v\in D^\prn{i}}$
are two distinct minimal projections, then the operators
${\phi^\prn{i}\prn{u\otimes 1_B}}$ and ${\phi^\prn{i}\prn{v\otimes 1_B}}$ are supported on disjoint sets. 
This should play a key role in bounding the generalised diagonal dimension from below with the asymptotic 
dimension.

\medskip

It should be stressed that this condition is automatically satisfied for $\C$ coefficients. 
That is,
the generalised diagonal dimension
extends the diagonal dimension of Li, Liao, and Winter.

\begin{alphpropo}
    Let ${(D_A \subseteq A)}$ be a nondegenerate \cstarSub{}, where $D_A$ is abelian.
    Then 
    \begin{equation*}
        \ddim{D_A}{A} = \ddimn{D_A}{A}{\mathbb{C}}.
    \end{equation*}
\end{alphpropo}

\bigskip

{We conjecture}
that for any uniformly locally finite metric space $X$,  
the generalised diagonal dimension of the noncommutative Cartan subalgebra in the \cstar{} generated by finite\myhyp{}propagation operators $\Cfp{X}$ is equal to the asymptotic dimension of $X$.

\begin{alphconj} \label{intro:thm:diagDimFPAlg}
     Let $X$ be a uniformly locally finite metric space, let $\Hcal$ be a separable, infinite\myhyp{}dimensional Hilbert space and set ${\Hcal_X \coloneq \ell^2(X, \Hcal)}$. 
     Then
    \begin{equation*}
         \ddimn{\CfpCartan{X}}{\Cfp{\Hcal_X}}{\ell^\infty(\N, \bop{\Hcal})} = \asdim{X}.
    \end{equation*}
\end{alphconj}

We can prove that 
\[\ddimn{\CfpCartan{X}}{\Cfp{\Hcal_X}}{\ell^\infty(\N, \bop{\Hcal})} \leq \asdim{X}\]
(Theorem~\ref{thm:upperBoundFPAlg}),
and provide a few steps towards a proof of the converse inequality.
However, we do not have a complete proof (for the converse inequality).

\medskip

We {also} expect that a similar result holds for the Roe algebra,
that is, 
\begin{equation*}
         \ddimn{\CroeCartan{X}}{\Croe{\Hcal_X}}{\ell^\infty(\N, \kop{\Hcal})} = \asdim{X},
\end{equation*}
when $X$ is a uniformly locally finite metric space.
However, this has yet to be proved.
One difficulty that arises in proving the above is that the coefficients
${\ell^\infty(\N, \kop{\Hcal})}$ do not form a von Neumann algebra,
while
{our proof-strategy for Conjecture}~\ref{intro:thm:diagDimFPAlg} uses that ${\ell^\infty(\N, \bop{\Hcal})}$ is a von Neumann algebra.

\bigskip
\noindent\textbf{Structure of the paper.}
The paper is structured as follows. 

In Section~\ref{section:prelim}, we cover some preliminaries on order zero maps, Cartan subalgebras, and Roe-like algebras.

We introduce the generalised diagonal dimension in Section~\ref{section:genDiagDim} and show that it extends the diagonal dimension of Li, Liao, and Winter. We also prove permanence properties of the generalised diagonal dimension and compare it with the nuclear dimension. 

{In Section~\ref{section:diagDimRoeAlg} we prove the upper bound of Conjecture~\ref{intro:thm:diagDimFPAlg} and explain a few steps of a tentative proof for the lower bound. 
Moreover, we include a concrete example, provided by an anonymous referee, showing that a key estimate in the proof-attempt in the first version of this paper 
fails.
}

\bigskip

\noindent\textbf{Acknowledgements.}
I would like to thank Federico Vigolo
for his guidance and support during and after my PhD project,
Ralf Meyer for his advice, 
and Diego Mart\'{i}nez and Wilhelm Winter for their helpful comments.
I am also grateful to the anonymous referee who pointed out a gap in the proof of a key estimate in the first version of the paper.
Finally, I would like to mention that I will not be pursuing this topic further because I have left academia.

This work was funded by the RTG 2491 \mydhyp{} \emph{Fourier Analysis and Spectral Theory} of the DFG.

\tableofcontents


\section{Preliminaries} \label{section:prelim}

In this paper, all \cstar{}s are assumed to be concrete, that is, they are faithfully represented on some Hilbert space $H$.

\subsection{Order zero maps}

Winter and Zacharias introduced order zero maps, a class of completely positive maps 
which preserve orthogonality~\cite{WinterZacharias2009orderZero}. 
\begin{defin}[{\cite[Definition~2.3]{WinterZacharias2009orderZero}}]
    Let $A$ and $B$ be \cstar{}s and ${\phi \colon A~\rAr~B}$ be a completely positive map. 
    We say that $\phi$ has \emph{order zero}, if
    \[\phi(a) \phi(a') =0,\]
    for each $0 \le a \in A$ and $0 \le  a' \in A$, satisfying $a a' =0$.
    
\end{defin}

Any $\ast$\myhyp{}homomorphism between \cstar{}s
is easily seen to be an order zero map and 
order zero maps are closely related to $\ast$\myhyp{}homomorphisms. 
Winter and Zacharias in~\cite{WinterZacharias2009orderZero} proved a structure theorem for order zero maps, and defined a $\ast$\myhyp{}homomorphism supporting such maps.

\begin{thm}[Central theorem of order zero maps, {\cite[Theorem~3.3]{WinterZacharias2009orderZero}}] \label{thm:orderZero}
    Let $A$ and $B$ be \cstar{}s and let ${\phi \colon A \rAr B}$ be an order zero map. 
    Set ${C\coloneq \calg{\phi(A)}\subseteq B}$. 
    Then there is a positive element 
    \[h \in \multAlg{C} \cap C',\]
    with ${\normreg{h} = \normreg{\phi}}$,
    where $C'$ is the 
    commutant\footnote{Note that if ${C \subseteq \bop{\Hcal}}$, then ${C' \coloneq \{ T \in \bop{\Hcal} \colon T c = cT  \, \forall \, c \in C\}}$.} 
    of $C$, 
    and a $\ast$\myhyp{}homomorphism
    \begin{align*}
        \pi_\phi \colon A~\rAr~\multAlg{C} \cap \{h\}',
    \end{align*}
    such that
    \begin{equation*}
        \phi(a) = h \pi_\phi (a) ,
    \end{equation*}
    for each $ a \in A$. 
    Moreover, if $A$ is unital, then $h = \phi(1_A)$.
\end{thm}

A $\ast$\myhyp{}homomorphism $\pi_\phi$ defined by the central theorem for the order zero map $\phi$ 
is called a \emph{supporting $\ast$\myhyp{}homomorphism} for $\phi$. 

\begin{rmk} \label{rmk:proofThmOrderZero}
   If $A$ is unital and ${C \subseteq \bop{\Hcal}}$ is nondegenerate, 
  then the supporting $\ast$\myhyp{}homomorphism
    \begin{align*}
        \pi_\phi \colon A~\rAr~\bop{\Hcal}
    \end{align*}
    of Theorem~\ref{thm:orderZero} is
    given by
    \begin{align*}
        \pi_{\phi} (a) = \solim_{n\to \infty} \Big( h + \frac{1}{n} 1_{\Hcal} \Big)^{-1} \phi(a) ,
    \end{align*} for each $a \in A$, where $\solim$ denotes the limit in $\bop{\Hcal}$ with respect to the strong operator topology (SOT).
\end{rmk}

\begin{rmk} \label{rmk:orderZeroSOT}
    Let ${(D_A \subseteq A)}$ and ${(D_B \subseteq B)}$ be nondegenerate \cstarSub{}s
    and
    suppose that
    $A$ is unital and ${1_A \in D_A}$.
    Moreover, suppose that 
    \[\phi \colon A~\rAr~B\] 
    is an order zero map, such that ${\phi(D_A) \subseteq D_B}$. 
    Assume that $B$ acts nondegenerately on the 
    Hilbert space $\Hcal$.
    By Remark~\ref{rmk:proofThmOrderZero},
    we see that
   { \begin{equation*}
        \pi_{\phi} (D_A) 
        \subseteq \overline{\calg{D_B, 1_{\Hcal}}}^{\mathrm{SOT}},
    \end{equation*}
    }
    where ${\overline{\calg{D_B, 1_{\Hcal}}}^{\mathrm{SOT}}}$ is the closure of ${\calg{D_B, 1_{\Hcal}}}$ in ${\bop{\Hcal}}$ with respect to the strong operator topology (SOT).
\end{rmk}

\smallskip

Using the central theorem for order zero maps, Winter and Zacharias defined a functional calculus for such maps: 

\begin{cor}[{Order zero functional calculus,~{\cite[Corollary~4.2]{WinterZacharias2009orderZero}}}] \label{cor:orderZeroCalc}
    Let $A$ and $B$ be \cstar{}s, ${\phi \colon A~\rAr~B}$ a contractive order zero map and {${0 \le f \in \contO{(0,1]}}$}. 
    Set ${C \coloneq \calg{\phi(A)} \subseteq B}$ and $h$, $\pi_{\phi}$ to be as in Theorem~\textup{\ref{thm:orderZero}}. 
    Then the map
    \begin{equation*}
        f(\phi) \colon A~\rAr~C \subseteq B
    \end{equation*}
    given by 
    \begin{equation*}
        f(\phi)(a) = f(h) \pi_{\phi}(a)\; \forall \, a \in A
    \end{equation*}
    is a well\myhyp{}defined completely positive order zero map. 
    
    Moreover, if ${\normreg{f}_{\contO{(0,1]}} \le 1}$, then $f(\phi)$ is contractive.
\end{cor}

\medskip
\subsection{Noncommutative Cartan subalgebras}

\begin{defin}
     ${a \in A}$ is called a 
    \emph{normaliser} of $D$ in $A$
    if ${a D a^\ast, a^\ast D a \subseteq D}$.
    We denote by $\normalisers{D}{A}$ the set of normalisers of $D$ in $A$. 
    The \cstarSub{} $D$ is called \emph{regular} if the set $\normalisers{D}{A}$ generates $A$ as a \cstar{}.
\end{defin}

Renault introduced the notion of Cartan subalgebras for \cstar{}s as maximal abelian, 
regular \cstarSub{}s with a faithful conditional expectation. 

\begin{defin}[{\cite{Renault2008cartanSub}}] \label{def:Cartan}
    Let ${(D \subseteq A)}$ be a \cstarSub{}. We say that $D$ is a \emph{Cartan subalgebra of} $A$ if:
    \begin{enumerate}\renewcommand{\labelenumi}{(\arabic{enumi})}
        \item $D$ is a maximal abelian $\ast$\myhyp{}subalgebra (\emph{masa}) of $A$,
        \item ${(D \subseteq A)}$ is nondegenerate, that is, $D$ contains an approximate unit for $A$,
        \item $D$ is regular,
        \item there exists a \emph{faithful conditional expectation} from $A$ onto $D$,
        that is, a completely positive contractive map ${E \colon A~\rAr~D}$
        with ${E|_D = \mathrm{Id}_D}$ and such that $E$ is injective on the set of positive elements of $A$.
    \end{enumerate}
    Moreover, if $D$ has the \emph{unique extension property} relative to $A$, that is,
    every pure state on $D$ extends uniquely to a pure state on $A$, then $D$ is said to be a \emph{\cstarDiag{}} in $A$.

    For a Cartan subalgebra $D$ in a \cstar{} $A$, we say that the inclusion ${D\subseteq A}$ forms a \emph{Cartan pair}.
\end{defin}

Exel {\cite{Exel2011noncommCartan}} introduced noncommutative Cartan subalgebras  
in separable \cstar{}s by generalising Renault's definition of Cartan subalgebras \cite{Renault2008cartanSub}. 
Building on Exel's work, Kwa\'{s}niewski and Meyer in \cite{KwasniewskiMeyer2020NoncommutativeCartan}
defined noncommutative Cartan subalgebras
without assuming separability.

Following Exel, we introduce virtual commutants in order to define noncommutative Cartan subalgebras.

\begin{defin}[{\cite{Exel2011noncommCartan}}]
    Let ${(D \subseteq A)}$ be a \cstarSub{}. 
    A \emph{virtual commutant} of $D$ in $A$ is a pair
    ${(\phi, J)}$ of a 
    closed
    2\myhyp{}sided ideal ${J \subseteq D}$
    and a linear map ${\phi \colon J~\rAr~A}$ such that
    \begin{equation*}
        \phi(d x) = d \phi(x) 
        \phantom{=} \text{and} \phantom{=}
        \phi(x d) = \phi(x) d,
    \end{equation*}
    for each $x \in J$ and $d \in D$.
\end{defin}

To obtain the definition of noncommutative Cartan subalgebras,
we replace the maximal abelian subalgebra condition from Definition~\ref{def:Cartan}
with a condition on the virtual commutants of the subalgebra.

\begin{defin}[{\cite{Exel2011noncommCartan, KwasniewskiMeyer2020NoncommutativeCartan}}]
    Let ${(D \subseteq A)}$ be a \cstarSub{}. We say that $D$ is a \emph{noncommutative Cartan subalgebra} of $A$ if 
    (2), (3), (4) from Definition~\ref{def:Cartan}, and
    \begin{enumerate}
        \item[(1')] ${\img{\phi} \subseteq D}$ for any virtual commutant ${(\phi,J)}$ of ${D \subseteq A}$
    \end{enumerate}
    hold.
    For a noncommutative Cartan subalgebra $D$ in a \cstar{} $A$, we say that the inclusion ${D\subseteq A}$ forms a \emph{noncommutative Cartan pair}.
\end{defin}

\medskip
\subsection{Roe-like algebras}

In this section, we define Roe algebras 
of discrete metric spaces.
In general, one can define Roe algebras of locally compact, second countable Hausdorff spaces
via geometric modules following
the classical approach of Higson and Roe \cite{HigsonRoe2000AnalyticKHomology},
see also \cite{WillettYu2020HigherIndexTheory} and \cite{MartinezVigolo2023RoeViaMods}.

Throughout this section, ${(X, \dist_X)}$ is a discrete metric space.
   We fix an infinite\myhyp{}dimensional, separable Hilbert space $\Hcal$ and set 
    \[{\Hcal_X \coloneq \ell^2(X, \Hcal)}.\]
    In the sequel, ${\Hcal_X}$ will denote the above Hilbert space for a given discrete metric space $X$, unless stated otherwise.
  For a subset ${U \subseteq X}$, we denote by the characteristic function $\indF_{U}$ of $U$, 
    the 
    projection operator ${\indF_{U} \in \bop{\Hcal_X}}$ given by
    \begin{equation*}
        \indF_{U}(f)(x) 
        =\begin{cases}
            f(x), &x\in U\\
            0, &x\notin U
        \end{cases},
    \end{equation*}
    for each ${f\in \Hcal_X}$.

We can define the support of operators on ${\Hcal_X}$.

\begin{defin}
    Let 
    ${T \in \bop{\Hcal_X}}$. 
    The \emph{support} of $T$, denoted by $\supp{T}$, 
    is the set of all ${(y,x) \in X \times X}$, such that
    \[\indF_V T \indF_U \ne 0\]
    for all open neighbourhoods $U$ of $x$ and $V$ of $y$.
\end{defin}

\begin{defin}
    We define the \emph{propagation} of an operator ${T \in \bop{H_X}}$ as
    \begin{equation*} 
        \prop{T} \coloneq \sup\{ \dist_X(y,x) \colon (y,x) \in \supp{T} \} \in [0, \infty].
    \end{equation*}
    We say that ${T \in \bop{H_X}}$ has \emph{finite propagation}
    if ${\prop{T} < \infty}$.
\end{defin}
    
We now have all the ingredients to define
the \cstar{} of finite\myhyp{}propagation operators.

\begin{defin}
    The \emph{$\ast$-algebra of operators of finite propagation} of $\Hcal_X$ is the $\ast$-algebra
    \begin{equation*}
        \Cfpsub{\Hcal_X} \coloneq \{ T \in \bop{\Hcal_X} \colon \prop{T} < \infty \},
    \end{equation*}
    and its norm closure, denoted by $\Cfp{\Hcal_X}$, is the \emph{\cstar{} of operators of finite propagation of} ${\Hcal_X}$.  
\end{defin}

\begin{rmk}
    Since ${(X, \dist_X)}$ is a discrete metric space,
    each operator ${T \in \bop{\Hcal_X}}$ has an infinite matrix representation given by
    \[T = (T_{x,y})_{x,y \in X} ,\]
    with each $T_{x,y} \in \bop{\Hcal}$. When ${T \in \bop{\Hcal_X}}$ has finite propagation, 
    then
    there exists ${r >0}$, such that
    ${T_{x,y} = 0}$ for each ${x,y \in X}$ with ${\dist_X(x,y) > r}$.
\end{rmk}

\begin{defin} 
    An operator
    ${T \in \bop{\Hcal_X}}$ is \emph{locally compact}
    if 
    \begin{equation*}
        \indF_K T,~T \indF_K \in \kop{\Hcal_X}
    \end{equation*}
    for each compact subset ${K \subseteq X}$.

    The \emph{Roe $\ast$-algebra} of $\Hcal_X$ is the $\ast$\myhyp{}algebra of all locally compact, finite\myhyp{}propagation operators,
    \begin{equation*}
        \Croesub{\Hcal_X} \coloneq \{ T \in \Cfpsub{\Hcal_X} \colon T \text{ is locally compact} \},
    \end{equation*}
    and its norm closure, denoted by ${\Croe{\Hcal_X}}$, is called the \emph{Roe \cstar{}} of $\Hcal_X$, or \emph{Roe algebra} of $\Hcal_X$.
\end{defin}

\smallskip

By taking the Hilbert space $\ell^2(X)$ instead of $\Hcal_X$, we can define the propagation of an operator in ${\bop{\ell^2(X)}}$ similarly.
Then the $\ast$\myhyp{}algebra of operators with finite propagation
\begin{equation*}
        \Cusub{X} \coloneq \{ T \in \bop{\ell^2(X)} ~\colon~\prop{T}<\infty \}
\end{equation*}
is called the \emph{uniform Roe $\ast$\myhyp{}algebra of} $X$, and its norm\myhyp{}closure, which is denoted by 
${\Cu{X}}$, is called the \emph{uniform Roe \cstar{}} of $X$, or \emph{uniform Roe algebra} of $X$.
In the infinite matrix representation picture, the operators 
\[T = (T_{x,y})_{x,y \in X}  \in \bop{\ell^2(X)}\] 
of finite propagation are operators in ${\bop{\ell^2(X)}}$ where ${T_{x,y} \in \C}$, 
and there exists $r >0$ such that
${T_{x,y} = 0}$, for each ${x,y \in X}$ with ${\dist_X(x,y) > r}$.
    
\smallskip

We refer to the Roe algebra, the \cstar{} of operators of finite propagation and the uniform Roe algebra
as \emph{Roe\myhyp{}like algebras},
using the naming convention of~\cite{MartinezVigolo2023RoeViaMods}.

\medskip
\subsubsection{Cartan subalgebras in Roe-like algebras}

Fix a uniformly locally finite metric space ${(X, \dist)}$.
Consider the \cstarSub{} of ${\Cfp{\Hcal_X}}$, given by
\begin{equation}
    {\CfpCartan{X} = \CfpCartan{\Hcal_X} \coloneq \{ T \in \Cfp{\Hcal_X}~\colon~\prop{T} = 0 \}.}
\end{equation}
The subalgebra {${\CfpCartan{X}}$} is 
equal to the image of the natural embedding of the algebra 
${\ell^\infty (X, \bop{\Hcal})}$ into ${\Cfp{\Hcal_X}}$,
defined by 
\begin{align*}
    \ell^\infty (X, \bop{\Hcal}) &~\rAr~\Cfp{\Hcal_X}\\
    (T_x)_{x \in X} &~\longmapsto~\, T \colon (\xi_x)_{x \in X}~\longmapsto~(T_x \xi_x)_{x \in X} .
\end{align*}
Similarly, we define the \cstarSub{} of $\Croe{\Hcal_X}$, given by
\begin{equation}
  {  \CroeCartan{X} =  \CroeCartan{\Hcal_X} \coloneq \{ T \in \Croe{\Hcal_X}~\colon~\prop{T} = 0 \},}
\end{equation}
which is equal to ${\ell^\infty (X, \kop{\Hcal})}$ when embedded into ${\Croe{\Hcal_X}}$. 

The \cstarPair{}s
{${(\CfpCartan{X}\subseteq\Cfp{\Hcal_X})}$ and ${(\CroeCartan{X}\subseteq\Croe{\Hcal_X})}$} form noncommutative Cartan 
pairs.

\begin{thm}[{\cite[Theorem 6.32]{MartinezVigolo2023RoeViaMods}}]
    Let ${(X, \dist)}$ be a uniformly locally finite metric space.
    Then {${\CfpCartan{X}}$} is a noncommutative Cartan subalgebra of ${\Cfp{\Hcal_X}}$ and
    {${\CroeCartan{X}}$} is a noncommutative Cartan subalgebra of ${\Croe{\Hcal_X}}$.
\end{thm}

The theorem above holds in more generality; however, in this paper, we focus our attention on uniformly locally finite metric spaces.
We refer the reader to~\cite[Section 6.4.]{MartinezVigolo2023RoeViaMods} for more details on noncommutative Cartan subalgebras in Roe\myhyp{}like algebras for coarse spaces of bounded geometry.

\bigskip\bigskip

\section{Generalised diagonal dimension} \label{section:genDiagDim}

Li, Liao, and Winter in \cite{LiLiaoWinter2023diagDim} introduced diagonal dimension, a version of nuclear dimension for diagonal \cstarSub{}s.

\begin{defin}[{Diagonal dimension,~\cite[Definition~2.1.]{LiLiaoWinter2023diagDim}}] \label{def:diagDim}
    Let ${(D_A \subseteq A)}$ be a C$^{\ast}$\myhyp{}subalgebra with $D_A$ abelian. 
    We say ${(D_A \subseteq A)}$ has \emph{diagonal dimension} at most $d$, written as
    \[\ddim{D_A}{A} \le d,\]
    if for any finite subset ${\Fcal \subseteq A}$ and ${\e > 0}$, 
    there exist a finite\myhyp{}dimensional 
    C$^{\ast}$-algebra ${F=F^{(0)} \oplus F^{(1)} \oplus \hdots \oplus F^{(d)}}$ with a maximal abelian $\ast$\myhyp{}subalgebra \linebreak
    ${D_F=D^{(0)} \oplus D^{(1)} \oplus \hdots \oplus D^{(d)}}$ 
    and completely positive maps 
    \[ A \overset{\psi}{\longrightarrow} F \overset{\phi}{\longrightarrow} A\]
    such that
    \begin{enumerate} \renewcommand{\labelenumi}{(\arabic{enumi})}
        \item $\psi$ is contractive,
        \item ${\normreg{\phi \circ \psi(a) - a } < \e} $ for every $a \in \Fcal$,
        \item ${\phi^{\prn{i}} \coloneq \phi|_{F^{(i)}}}$ is a contractive order zero map for each ${i \in  \{0,\hdots,d\}}$,
        \item ${\psi(D_A) \subseteq D_F}$,
        \item ${\phi^{\prn{i}} \big( \normalisers{D^{(i)}}{F^{(i)}}\big) \subseteq \normalisers{D_A}{A}}$.
    \end{enumerate}
\end{defin}

The condition for the normalisers in the definition above can be written 
equivalently in terms of matrix units.
We may replace condition (5) with the following equivalent condition:
\begin{enumerate}
        \item[(5')] $\phi$ maps every matrix unit with respect to $D_F$ into $\normalisers{D_A}{A}$, that is, if ${v\in F}$ such that ${v v^\ast}$ and ${v^\ast v}$ are minimal projections in ${D_F}$, then ${\phi(v) \in \normalisers{D_A}{A}}$.
    \end{enumerate}
The two conditions are equivalent because for finite\myhyp{}dimensional \cstar{}s 
the normalisers of a \cstarDiag{} can be written as a linear combination of pairwise orthogonal matrix units, see e.g.~\cite[Example 2]{Kumjian1986onCDiagonals}.

\begin{rmk} \label{rmk:diagDimVsNucDim}
    By dropping Conditions~(4) and (5) from the definition of the diagonal dimension,
    we obtain the nuclear dimension defined by Winter and Zacharias in \cite{WinterZacharias2010nucDim}.
Therefore, for each nondegenerate \cstarSub{} ${(D_A \subseteq A)}$, we have
\begin{equation*}
    \dimnuc{A} \le \ddim{D_A}{A}.
\end{equation*}
In particular, non\myhyp{}nuclear \cstar{}s always have infinite diagonal dimension.
\end{rmk}

\subsection{Generalised diagonal dimension}

We introduce a generalisation of the diagonal dimension that can take finite values also for non\myhyp{}nuclear \cstar{}s.

\begin{defin}[Generalised diagonal dimension] \label{def:genDiagDim}
    Fix a \cstar{} $B$. 
    Let \linebreak
    ${(D_A\subseteq{A})}$ be a nondegenerate \cstarSub{}. 
    We say ${(D_A \subseteq A)}$ has \emph{generalised diagonal dimension with respect to $B$} at most $d$, written as
    \[\ddimn{D_A}{A}{B} \le d,\]
    if for any finite subset 
    ${\Fcal \subseteq {A}}$ and ${\e > 0}$, there exist a finite\myhyp{}
    \linebreak 
    dimensional
    \cstar{} 
    ${F=F^{(0)}\oplus F^{(1)} \oplus \hdots \oplus F^{(d)}}$ 
    with canonical diagonal \linebreak
    ${D_F = D^{(0)} \oplus D^{(1)} \oplus \hdots \oplus D^{(d)}}$ 
    and completely positive maps 
    \[ {A \overset{\psi}{\longrightarrow} F \otimes B \overset{\phi}{\longrightarrow} A}\]
    such that
    {
    \begin{enumerate} \renewcommand{\labelenumi}{(\arabic{enumi})}
        \item $\psi$ is contractive,
        \item ${\normreg{\phi \circ \psi(a) - a } < \e }$ for every $a \in \Fcal$,
        \item ${\phi^{(i)} \coloneq \phi|_{F^{(i)} \otimes B}}$ is a contractive order zero map for each ${i \in \{0,\hdots,d\}}$,
        \item ${\psi(D_A) \subseteq D_F \otimes B}$,

        \item ${\phi^{(i)} \big( {D^{(i)} \otimes B} \big) \subseteq \normalisers{D_A}{A}}$
        and
        ${\phi^{(i)} ( v \otimes B )  \subseteq \normalisers{D_A}{A}}$
        for each matrix unit ${v \in F^\prn{i}}$ with respect to ${D^{\prn{i}}}$,
        
        \item for each ${i\in \{0,\hdots, d\}}$, if
        $\pi^{(i)}$ is a supporting ${\ast}$\myhyp{}homomorphism for the order zero map $\phi^{(i)}$ and ${\{u_\lambda\}_\lambda}$ is an approximate unit of $B$, then
        \[{\solim_\lambda \pi^{(i)}(v \otimes u_\lambda) \in (D_A)'},\]
        for each minimal projection ${v \in D^{\prn{i}}}$.
    \end{enumerate}
    }
\end{defin}

We say that ${(F,~D_F,~\psi,~\phi)}$ is a \emph{completely positive (c.p.) approximation witnessing 
${\ddimn{D_A}{A}{B} \le d}$ for $\Fcal$ within $\e$} if the tuple ${(F,~D_F,~\psi,~\phi)}$ satisfies the conditions in Definition~\ref{def:genDiagDim} with respect to the finite subset ${\Fcal \subseteq A}$ and ${\e>0}$. 

\smallskip

It can be proved that the completely positive maps in Definition~\ref{def:genDiagDim} can be chosen such that ${\phi \circ \psi}$ is contractive.
For later use, we record below a simple lemma. Its proof is explained in \cite[Remark~2.2.(ii)]{LiLiaoWinter2023diagDim} and \cite[Remark~2.2.(iv)]{WinterZacharias2010nucDim}.

\begin{lem} \label{lem:cpPhioPsiContractive}
    Suppose that $A$ is unital and ${\ddimn{D_A}{A}{B} = d < \infty}$. Then we can choose a completely positive approximation 
    such that ${\phi \circ \psi}$ is contractive.
\end{lem}

\begin{rmk}
      The lemma above can be extended for non\myhyp{}unital \cstar{}s by using an approximate unit, see~\cite[Remark 2.2.(iv)]{WinterZacharias2010nucDim} for similar techniques.
\end{rmk}

\subsubsection{Commutation property}

We introduce the following property. 

\begin{defin}[Commutation property CoP] \label{def:PropP}
    Let $B$ be a unital C$^{\ast}$-algebra, $\Hcal$ be a Hilbert space and ${\pi \colon M_r(B)~\rAr~\bop{\Hcal}}$ be a ${\ast}$\myhyp{}homomorphism with ${r \in \mathbb{N}}$. 
    Suppose that ${D \subseteq \bop{\Hcal}}$ is a \cstar{}.
    
    Then we say that $\pi$ has the \emph{commutation property for $D$} (\emph{\CoProp{D}}) 
    if for each minimal projection ${v \in D_r(\C)}$ 
    we have
    \[{d \pi(v\otimes 1_B)  =\pi(v\otimes 1_B)  d},\]
    for each ${d \in D}$.
    
    When $B$ is not unital, we fix an (increasing) approximate unit ${\{u_\lambda\}_{\lambda}}$ of $B$.
    We say that $\pi$ has \emph{\CoProp{D}} if
    for each minimal projection ${v \in D_r(\C)}$ and each ${d \in D}$ it holds
    \[d  \big(\solim_\lambda \pi(v \otimes u_\lambda) \big) = \big( \solim_\lambda \pi(v \otimes u_\lambda)  \big) d,\]
    where $\solim$ denotes the limit in the strong operator topology (SOT).

    We say that ${\pi \colon M_{r_1}(B)  \oplus \hdots \oplus  M_{r_m}(B)~\rAr~\bop{\Hcal}}$ has \emph{\CoProp{D}} if for each ${i\in \{1,\hdots,m\}}$ the restriction ${\pi|_{M_{r_i}(B)}}$ has \CoProp{D}.
\end{defin}

Using the above property we can rephrase Condition~(6) of Definition~\ref{def:genDiagDim} with the following condition: 
for each ${i\in \{0,\hdots, d\}}$, if
$\pi^{(i)}$ is a supporting ${\ast}$\myhyp{}homomorphism 
for the order zero map $\phi^{(i)}$, 
then $\pi^{(i)}$ has \CoProp{D_A}.

\begin{rmk}
    This commutation property is of interest
    because in the sequel, we will define order zero maps whose supporting $\ast$\myhyp{}homomorphisms 
    map the diagonal matrix units to operators supported in pairwise disjoint sets.
    Maps into \cstar{}s with CoP for their noncommutative Cartan subalgebra have this property.
\end{rmk}

\begin{example}
    Let ${\Hcal}$ be an infinite\myhyp{}dimensional, separable Hilbert space and define
    the $\ast$\myhyp{}homomorphism ${\pi}$ given by
    \[\begin{array}{ccccc}
        {\pi} &\colon & M_2(\bop{\Hcal}) &~\longrightarrow~& \Cfp{\ell^2(\N, \Hcal)}\\
        \\
         & & \begin{pmatrix}
            T_{1,1}  & T_{1,2}  \\
            T_{2,1}  & T_{2,2}
        \end{pmatrix}
        &~\longmapsto~& \begin{pmatrix}
            T_{1,1} & T_{1,2} & 0       & 0  & \hdots \\
            T_{2,1} & T_{2,2} & 0       & 0  \\
            0       & 0       & T_{1,1} & T_{1,2}  \\
            0       & 0       & T_{2,1}  & T_{2,2} \\
            \vdots & &  & & \ddots 
        \end{pmatrix} .
    \end{array}\]
    Then ${\pi}$ has \CoProp{\CfpCartan{\ell^2(\N, \Hcal)}}
    and,
    moreover, maps diagonal matrix units to operators with pairwise disjoint supports.
    
    Conversely,
    we can define a
     $\ast$\myhyp{}homomorphism 
    \[\begin{array}{ccccc}
        \widetilde{\pi} & \colon & M_2(\bop{\Hcal}) 
        &~\rAr~& 
        \Cfp{\ell^2(\N, \Hcal)}
    \end{array}\]
    that does not have either of these properties:
    Let ${p, q \in \bop{\Hcal} \backslash \{1_{\Hcal} , 0\}}$ be two infinite\myhyp{}rank projections such that ${q = 1_{\Hcal} - p }$. 
    We can define isometries ${v_1, v_2 \in \bop{\Hcal}}$, satisfying ${v_1 \Hcal = p\Hcal}$ and ${v_2 \Hcal = q\Hcal}$.
    Set
    \[\begin{array}{ccccc}
        v \coloneq v_1 \oplus v_2& \colon & \Hcal  \oplus  \Hcal & \rAr & \Hcal\\
        & & (h_1, h_2 ) & \longmapsto &  v_1 h_1 + v_2 h_2 
    \end{array}.\] 
    Let ${\widetilde{\pi}}$
    be the $\ast$\myhyp{}homomorphism given by
    \[\begin{array}{ccccc}
       \widetilde{\pi} & \colon & M_2(\bop{\Hcal}) 
        &~\rAr~& 
        \Cfp{\ell^2(\N, \Hcal)}\\
        \\
        &  & T \phantom{\Hcal))}
        &~\longmapsto~&
        \begin{pmatrix}
            v T v^\ast  & 0             & 0 & \hdots\\
            0           & v T v^\ast    & 0 &  \\
            0           & 0             & v T v^\ast     &  \\
            \vdots       &     &  & \ddots 
        \end{pmatrix} .
    \end{array}\]
    It is straightforward to see that
    \[\widetilde{\pi} \begin{pmatrix}
        1_{\Hcal} & 0 \\
        0 & 0 
    \end{pmatrix}  \notin \big(\CfpCartan{\ell^2(\N, \Hcal}\big)'\]
    and, hence,
    ${\widetilde{\pi}}$ does not have 
    \CoProp{\CfpCartan{\ell^2(\N, \Hcal)}}.
    Moreover, 
    we have that the operators
   \[\widetilde{\pi} \begin{pmatrix}
        1_{\Hcal} & 0 \\
        0 & 0 
    \end{pmatrix} = \begin{pmatrix}
            p  & 0             & 0 & \hdots\\
            0           & p    & 0 &  \\
            0           & 0             & p    &  \\
            \vdots       &     &  & \ddots 
        \end{pmatrix} 
        \phantom{==} \text{ and } \phantom{==} 
        \widetilde{\pi} \begin{pmatrix}
        0 & 0 \\
        0 & 1_{\Hcal}
    \end{pmatrix} = \begin{pmatrix}
            q  & 0             & 0 & \hdots\\
            0           & q    & 0 &  \\
            0           & 0             & q    &  \\
            \vdots       &     &  & \ddots 
        \end{pmatrix}\]
        are supported on the same set.
        Hence,
        ${\widetilde{\pi}}$ maps the diagonal matrix units to operators supported on the same set.
\end{example}

\medskip
\subsubsection{Comparison with the diagonal dimension}

The generalised diagonal dimension extends the diagonal dimension of Li, Liao, and Winter (Definition~\ref{def:diagDim}).
More precisely, if $D_A$ is an abelian \cstarSub{} in a \cstar{} $A$, then
the generalised diagonal dimension with respect to $\C$ is equal to the diagonal dimension.

\begin{propo} \label{prop:genDiagDimExtendsDiagDim}
    Let ${(D_A \subseteq A)}$ be a nondegenerate \cstarSub{}, where $D_A$ is abelian.
    Then 
    \begin{equation*}
        \ddim{D_A}{A} = \ddimn{D_A}{A}{\mathbb{C}}.
    \end{equation*}
\end{propo}

\begin{proof}
    The inequality $\le$ is automatic. We elaborate on the converse
    inequality.

    From the discussion after Definition~\ref{def:diagDim}, we see that Condition~(5) 
    from Definition~\ref{def:genDiagDim} for $B=\C$
    is equivalent to Condition~(5) in Definition~\ref{def:diagDim}. 
    Therefore, it is only left to show that Condition (6) is automatic when $D_A$ is abelian and $B=\C$.
    
    If ${\phi \colon M_r(\C) \rAr A}$ and ${\phi(D_r(\C)) \subseteq D_A}$, 
    then by the central theorem of order zero maps (Theorem~\ref{thm:orderZero}) 
    we can define a supporting  $\ast$\myhyp{}homomorphism 
    \[\pi \colon M_r(\C)~\rAr~\bop{H}\]
    such that ${\phi\prn{\cdot} = h \pi \prn{\cdot}}$
    and
    ${h \in \big(\img{\pi}\big)'}$. 
    Remark~\ref{rmk:orderZeroSOT} implies
    \[\pi(D_r(\C)) \subseteq \overline{\calg{D_A, \mathbf{1}_{\bop{H}}}}^{\mathrm{SOT}}, \]
    where:
    \begin{itemize}
        \item ${\calg{D_A, \mathbf{1}_{\bop{H}}} }$ is the \cstar{} generated  by $D_A$ and ${\mathbf{1}_{\bop{H}}}$,
        \item ${\overline{\calg{D_A, \mathbf{1}_{\bop{H}}}}^{\mathrm{SOT}}}$ 
        denotes the closure of ${\calg{D_A, \mathbf{1}_{\bop{H}}}}$ in $\bop{H}$ with respect to the strong operator topology (SOT).
    \end{itemize}

    By the von Neumann double commutant theorem, we have
    \[ \overline{\calg{D_A, \mathbf{1}_{\bop{H}}}}^{\mathrm{SOT}} =  \big({\calg{D_A, \mathbf{1}_{\bop{H}}}}\big)'' = {D_A}'',\] where ${D_A}''$ is the double commutant of ${D_A}$.
    Combining the above we obtain
     \[{\pi(D_r(\mathbb{C})) \subseteq D_A{''}}.\] 

    {Since ${D_A}$ is abelian, then ${D_A''\subseteq D_A'}$, 
    and, hence, ${\pi(D_r(\mathbb{C})) \subseteq D_A{'}}$.}
    Therefore, $\pi$ has \CoProp{D_A}, in other words,
    Condition~(6) of Definition~\ref{def:genDiagDim} is satisfied. \qedhere  
\end{proof}

\medskip
\subsection{Permanence properties and nuclear dimension}

The following permanence properties for the generalised diagonal dimension hold.

\begin{propo} \label{prop:permProp}
    Let ${(D_A \subseteq A)}$ and ${(D_C \subseteq C)}$ be nondegenerate \cstarSub{}s.
    Then we have the following properties.
   { \begin{enumerate} \renewcommand{\labelenumi}{(\roman{enumi})}
        \item Direct sums:
    \begin{align*}
        &\ddimn{D_A\oplus D_C}{A \oplus C }{B} \\
       &~\le \max \{ 
       \ddimn{D_A}{A}{B}, 
       \ddimn{D_C}{C}{B} \} .
    \end{align*}
        \item Tensor products:
        \begin{align*}
       &\ddimn{D_A\otimes D_C}{A \otimes C }{B} +1 \\
       &~\le \big( \ddimn{D_A}{A}{\C}+1\big) \big(\ddimn{D_C}{C}{B} +1 \big).
    \end{align*}
    In particular, for each $n \in \N$
    \begin{align*}
        \ddimn{D_n(D_A)}{M_n(A)}{B} \le \ddimn{D_A}{A}{B}.
    \end{align*}
    \end{enumerate}
}
\end{propo}

These permanence properties can be proved by making the natural choices
for the completely positive maps $\phi$ and $\psi$
(direct sums of approximations for (i), and tensor products for (ii)). 
A detailed proof can be found in \cite{Kitsios2025PhDThesis}. 
For similar ideas we refer the reader to the permanence properties
of the nuclear dimension~\cite{WinterZacharias2010nucDim} and 
diagonal dimension~\cite{LiLiaoWinter2023diagDim}.

\smallskip

Using Remark~\ref{rmk:diagDimVsNucDim},
we observe that 
for each nondegenerate \cstarSub{} 
\linebreak
${(D_A \subseteq A)}$, we have
\begin{equation*}
    \dimnuc{A} \le \ddim{D_A}{A} \le \ddimn{D_A}{A}{\C}.
\end{equation*}
Moreover, when $B$ is chosen to be a \cstar{} of finite nuclear dimension,
one can obtain a bound for the nuclear dimension of $A$ 
involving the generalised diagonal dimension.
 
\begin{thm}
    Let ${(D_A \subseteq A)}$ be a nondegenerate \cstarSub{} and let $B$ be a \cstar{} with finite nuclear dimension. 
    If ${\ddimn{D_A}{A}{B}< \infty}$, then $A$ is a nuclear \cstar{} with finite nuclear dimension given by
    \begin{equation*}
        \dimnuc{A} + 1 \le \big(\dimnuc{B}+1\big)\big(\ddimn{D_A}{A}{B}+1\big).
    \end{equation*}
\end{thm}

One can prove the above theorem by defining
completely positive maps for the nuclear dimension of $A$
as direct sums of completely positive approximations for ${\dimnuc{B}}$ and ${\ddimn{D_A}{A}{B}}$, 
see \cite{Kitsios2025PhDThesis} for details.

\bigskip
\section{Diagonal dimension of Roe\myhyp{}like algebras} \label{section:diagDimRoeAlg}

In this section, we 
{attempt to}
compute the diagonal dimension of the inclusion of the noncommutative Cartan subalgebra in the \cstar{} of finite\myhyp{}propagation operators on a uniformly locally finite metric space.

Before computing the diagonal dimension we define the asymptotic dimension for metric spaces. 

\begin{defin} \label{def:asDim}
    Let ${(X, \dist)}$ be a metric space.
    Then
$X$ has \emph{asymptotic dimension ${d \ge 0}$}, if $d$ is the least integer,
with the following property:
for any ${r>0}$, there exists a covering ${\Ucal = \{U_j\}_{j \in J}}$ of $X$,
with a decomposition 
\[\Ucal = \Ucal^{\prn{0}} \sqcup \hdots \sqcup \Ucal^{\prn{d}},\]
such that
$\Ucal$ is uniformly bounded, i.e.
\[\sup_{U \in \Ucal} \diam{U}< \infty, \]
and for each ${i \in \{0,\hdots,d\}}$ the family
${\Ucal^{\prn{i}}}$ is $r$\myhyp{}separated, i.e.
${\dist(U_{j} , U_{j'}) > r}$ for each ${U_{j} , U_{j'} \in \Ucal^{\prn{i}}}$ with ${j \ne j'}$. 
\end{defin}

The asymptotic dimension was introduced by Gromov~\cite{Gromov1993AsymptoticInvariantsInfiniteGroups}
for finitely generated groups and 
was later extended by Roe~\cite{Roe2003lecturesCrsGeometry} for a larger class of spaces, including coarse spaces.

\smallskip

We prove that the generalised diagonal dimension in the \cstar{} of finite\myhyp{}propagation operators is {bounded by} the asymptotic dimension of the space.

\begin{thm} \label{thm:upperBoundFPAlg}
    Let $X$ be a uniformly locally finite metric space.
    Then
    \begin{align*}
         \ddimn{\CfpCartan{X}}{\Cfp{\Hcal_X}}{\ell^\infty(\N, \bop{\Hcal})} \le \asdim{X}.
    \end{align*}
\end{thm}

This theorem is obtained by extending the proof of the corresponding inequality in \cite[Theorem~7.7]{LiLiaoWinter2023diagDim}. 
It is worth noting that \cite[Theorem~7.7]{LiLiaoWinter2023diagDim} in turn follows from the proof of \cite[Theorem~8.5]{WinterZacharias2010nucDim}, which states that
\begin{equation*}
    \dimnuc{\Cu{X}} \le \asdim{X}.
\end{equation*}
We only sketch the proof of the theorem here and refer the reader to a detailed proof found in \cite[Theorem~5.2]{Kitsios2025PhDThesis}.

\begin{proof}[Sketch of proof]
    Suppose that $d \coloneq \asdim{X} < \infty$.
    Let $\Fcal \subseteq \Cfp{\Hcal_X}$ be a finite set and $\e>0$.

    Following \cite[Theorem~7.7]{LiLiaoWinter2023diagDim} we may fix ${r \in \N}$ large enough and choose
     ${(d+1)}$ uniformly bounded 
     $3r$\myhyp{}separated families ${\mathcal{U}^{(0)},\hdots,\mathcal{U}^{(d)}}$
     covering $X$, i.e.
     \[X= \bigcup_{i=0}^d \mathcal{U}^{(i)}.\]

     Let
     \begin{equation*}
         B^{(i)} \coloneq \prod_{U \in \mathcal{U}^{(i)}} \indF_{\ball{}{U}{r}} \,  \Cfp{\Hcal_X} \, \indF_{\ball{}{U}{r}},
     \end{equation*}
     and
     \begin{equation*}
         D^{(i)} \coloneq  \prod_{U \in \mathcal{U}^{(i)}} \indF_{\ball{}{U}{r}} \,  D_A \, \indF_{\ball{}{U}{r}},
     \end{equation*}
    and set
      \begin{align*}
        f_i \coloneq \frac{1}{r} \sum_{U \in \mathcal{U}^{(i)}} \sum_{m=1}^{r} \indF_{\ball{}{U}{m}},
    &&
        f \coloneq \sum_{i=0}^d f_i,
   && \text{and}
   &&
        h_i \coloneq \big( f_i f^{-1} \big)^{1/2}, 
    \end{align*}
    for each ${i \in \{0,\hdots,d\}}$.

    We define completely positive maps
    \begin{equation} \label{eq:upperBoundAsdimApprox1}
        A 
        \xlongrightarrow{  \Psi_r}  
       B^{(0)} \oplus B^{(1)} \oplus \hdots \oplus B^{(d)} 
        \xlongrightarrow{ \Phi_r } 
         A
    \end{equation}
    given by
    \[\begin{array}{cccccccccc}
         \Psi_r \colon 
         & A 
         &~\rAr~
         & B^{(0)} 
         & \oplus 
         & B^{(1)} 
         & \oplus & \hdots & \oplus 
         & B^{(d)}\\
         & T 
         &~\longmapsto~
         & (h_0 T^{\prn{0}} h_0
         &  ,
         & h_1 T^{\prn{1}} h_1 
         & ,
         & \hdots & ,
         & h_d T^{\prn{d}} h_d )
     \end{array}\]
     and
         \[ \begin{array}{cccccccccc}
         \Phi_r  \colon & B^{(0)} & \oplus & B^{(1)} & \oplus & \hdots & \oplus & B^{(d)} &~\rAr~& A \\
        & (a_0 & , & a_1 & , & \hdots & , & a_d)
        &~\longmapsto~& a_0 + a_1 + \hdots + a_d    
     \end{array}.
\]
    We see that ${\Psi_r}$ and ${\Phi_r}$ satisfy the following:
    \begin{enumerate} \renewcommand{\labelenumi}{(\alph{enumi})}
        \item \label{eq:upperBoundCPAlgPropa} $\Psi_r$ is a completely positive contractive map,
        \item \label{eq:upperBoundCPAlgPropb} ${\normreg{\Phi_r \circ \Psi_r(T) - T} < \e}$ for each ${T \in \Fcal}$,
        \item \label{eq:upperBoundCPAlgPropc}$\Phi_r|_{B^{(i)}}$ is a contractive order zero map for each ${i \in \{0,\hdots,d\}}$,
        \item \label{eq:upperBoundCPAlgPropd} ${\Psi_r(D_A) \subseteq \bigoplus_{i=0}^d D^{\prn{i}}}$,
        \item \label{eq:upperBoundCPAlgPrope} ${\Phi_r|_{B^{(i)}} \big( \normalisers{D^{(i)}}{B^{(i)} } \big) \; \subseteq  \; \normalisers{D_A }{ A}}$ for each ${i \in \{0,\hdots,d\}}$.
    \end{enumerate}
    It should be stressed that the property (b) above is only true when $r$ is large enough.
    
    The maps ${\Psi_r}$ and ${\Phi_r}$ induce completely positive maps
    \begin{equation*}
        \Cfp{\Hcal_X} 
        \xlongrightarrow{ \psi}  
       \bigoplus_{i=0}^d F^{(i)} \otimes \ell^\infty(\N, \bop{\Hcal})
        \xlongrightarrow{\phi } 
         \Cfp{\Hcal_X},
    \end{equation*}
    where ${F^{(i)}}$ is a finite\myhyp{}dimensional \cstar{} for each ${i \in \{0, \hdots, d\}}$.

    Properties \hyperref[eq:upperBoundCPAlgPropa]{(a)},
    \hyperref[eq:upperBoundCPAlgPropb]{(b)},
    \hyperref[eq:upperBoundCPAlgPropc]{(c)},
    \hyperref[eq:upperBoundCPAlgPropd]{(d)}
    and \hyperref[eq:upperBoundCPAlgPrope]{(e)} imply
    Conditions (1), (2), (3), (4) and (5) of Definition~\ref{def:genDiagDim},
    respectively.
    Moreover, it is straightforward to show that Condition~(6) holds by using the definition of ${\Phi_r}$. \qedhere
\end{proof}

\medskip
\subsection{Lower bound}

{We pose the following.}
\begin{conj}  \label{conj:lowerBoundFPAlg}
    Let $X$ be a uniformly locally finite metric space.
    Then
    \begin{align*}
         \ddimn{\CfpCartan{X}}{\Cfp{\Hcal_X}}{\ell^\infty(\N, \bop{\Hcal})} \ge \asdim{X}.
    \end{align*}
\end{conj}

The above together with Theorem~\ref{thm:upperBoundFPAlg} would prove Conjecture~\ref{intro:thm:diagDimFPAlg}.

The proof-strategy we intended to follow is to extend the ideas from the proof of the
lower bound in \cite[Theorem~7.7]{LiLiaoWinter2023diagDim} 
and to work with operator-valued functions in $\ell^\infty(X,\mathcal B(\mathcal H))$ 
instead of complex-valued functions in $\ell^\infty(X)$.
With the above method and using techniques inspired by the proof of Theorem~6.4~in~\cite{GuentnerWillettYu2017dynAsDim}
instead of working with groupoids and the dynamic asymptotic dimension,
we should be able to bound the diagonal dimension directly by the asymptotic dimension. 
This strategy goes as follows.

Suppose that $(X, \dist_X)$ is a uniformly locally finite metric space. 
For brevity, we denote the \cstarPair{} by 
    \[\big( D_A \subseteq A \big) \coloneq \big(\CfpCartan{X} \subseteq \Cfp{\Hcal_X} \big)\]
and 
\[B \coloneq \ell^\infty(\N, \bop{\Hcal}).\]
Assume that 
\[\ddimn{D_A}{A}{B}  = d < \infty.\]

Fix ${r>0}$.
{We need to construct a uniformly bounded covering 
\[\Ucal = \Ucal^{\prn{0}} \sqcup \hdots \sqcup \Ucal^{\prn{d}},\]
where each the family $\Ucal^{(i)}$ is $r$-separated.}

Set 
\begin{equation*}
    E_0 \coloneq \big\{ (x,y)\in X \times X ~ \colon ~ \distX{x}{y} \le r \big\}.
\end{equation*}
Since $X$ is uniformly locally finite, there exist disjoint sets ${S_1, S_2, \hdots, S_M \subseteq X \times X}$ such that
\[{E_0 = S_1 \sqcup S_2 \sqcup \hdots \sqcup S_M},\]
  and for each ${m\in \{1,\hdots, M \}}$ the (range and source) maps 
    \[\begin{array}{ccc}
        S_m & \rAr & X\\
        (x,y) & \longmapsto &  x
    \end{array}\] 
  and
  \[\begin{array}{ccc}
        S_m & \rAr & X\\
         (x,y) & \longmapsto &  y
    \end{array}\] 
are injective.

For each ${m \in \{1,\hdots , M\}}$ define the operator ${a_m \in \Cfp{\Hcal_X}}$, with the matrix representation
${a_m = \big( a_m(x,y) \big)_{x,y \in X}}$
given by
\begin{equation*}
        a_m (x,y) \coloneq \begin{cases}
            \id{\Hcal} , & (x,y) \in S_m \\
            0, & (x,y) \notin S_m
        \end{cases}.
\end{equation*}
    
Define the finite set
\begin{equation*}
        \Fcal \coloneq \big\{ 1_A, \, a_1 ,\hdots, \, a_M \big\} \subseteq A.
\end{equation*}
Choose a completely positive approximation ${(F,~D_F,~\psi,~\phi)}$
witnessing 
\[\ddimn{D_A}{A}{B}=d\]
for 
{${(\Fcal \cup \Fcal^2 )^2}$ 
}
within {a small enough constant $\e>0$},
where
\[F=F^{(0)} \oplus F^{(1)} \oplus \hdots \oplus F^{(d)}\]
is a finite\myhyp{}dimensional \cstar{} with the canonical diagonal 
\[D_F=D^{(0)} \oplus D^{(1)} \oplus \hdots \oplus D^{(d)}.\]
Note that for each $i$, the \cstar{} $F^{\prn{i}}$ is finite dimensional,
hence,
\begin{equation*}
        F^{\prn{i}} = \bigoplus_{j=1}^{N^\prn{i}} M_{n^{(i),j}}(\C),
\end{equation*}
for some $N^{\prn{i}}, n^{(i),j}\in \N$.

We fix the \emph{generalised matrix units} given by
\begin{equation} \label{eq:genMatrixUnits}
       \Big( e_{k,l}^{(i),j} \Big)_{k,l=1}^{s^{(i),j}} : = \Big( e_{k,l}^{(i),j,\mathbb{C}} \otimes 1_B \Big)_{k,l=1}^{s^{(i),j}},
\end{equation}
where $e_{k,l}^{(i),j,\mathbb{C}}$ are the canonical matrix units in $ M_{s^{(i),j}}(\mathbb{C})$. 
    
For each $i,j$ we define the order zero map
\begin{equation} \label{eq:orderZeroMaps(i,j)}
        \begin{array}{ccccc} 
            \phi^{(i),j} \coloneq \phi\vert_{M_{s^{(i),j}}(B)} & \colon & M_{s^{(i),j}}(B) & \rAr & A \\
            & & b &\longmapsto & \phi(b)  
        \end{array}
\end{equation}
to be the restriction of $\phi^\prn{i}$ to $ M_{s^{(i),j}}(B) \subseteq F^{(i)} \otimes B$.

The idea is to use ``approximate supports'' of the images
of $e_{k,k}^{(i),j}$ 
under an appropriate rescaling of $\phi$. Specifically, one can define
\[\unMatImg{i}{j}{k,k}\coloneq f(\phi^{(i),j})(e_{k,k}^{(i),j})\]
for a function $f$ defined as in \cite{LiLiaoWinter2023diagDim} 
and, then,
\begin{equation}\label{def:USets}
       {U^{(i),j}_k} \coloneq \Big\{ x \in X  \colon  \normbig{
        \unMatImg{i}{j}{k,k}
         \indF_x \unMatImg{i}{j}{k,k} }> \eta^2 \Big\},
\end{equation}
where $\eta>0$ is sufficiently small.
With a certain amount of work, it can be shown 
that:
\begin{equation}\label{clm:covering}
    \textit{The collection ${\mathcal{C} \coloneq \big({U^{(i),j}_k}\big)_{i,j,k}}$ is a cover of $X$.}
\end{equation}

For each $i\in \{0,\hdots,d\}$ set  
    \begin{equation*}
        U^{(i)} \coloneq \bigcup_{j=1}^{r^{(i)}} \bigcup_{k=1}^{s^{(i),j}}  U^{(i),j}_k . 
    \end{equation*}
    Define the equivalence relation\footnote{This was inspired by the equivalence relations defined in 
   ~\cite[Theorem~6.4]{GuentnerWillettYu2017dynAsDim} and~\cite[Proposition~6.6]{LiLiaoWinter2023diagDim}.} 
   on $U^{(i)}$
    given by:
    \begin{equation*}
        x \sim_i y
    \end{equation*}
    if and only if there are 
    $x_1,\hdots,x_\kappa \in U^{(i)}$  
    such that:
    \begin{itemize}
        \item $x_1 = x,$
        \item {$ x_{\kappa} = y$},
        \item ${\distX{x_l}{x_{l+1}} \le r}$ for each $l \in \{1,\hdots,\kappa-1\}$.
    \end{itemize}
    
    Assuming \eqref{clm:covering}, we claim that in order to bound the asymptotic dimension, 
    it suffices to prove that the cardinalities
    of the equivalence classes ${[x]_{i}}$ are uniformly bounded.

    \begin{lem}\label{lem:LowerBoundProof1}
    If
    \begin{equation*}
        S\coloneq \max_{i=0,\hdots,d} \bigg(\sup_{x \in U^{(i)}} \big{\lvert}{ [ x ]_{i}}   \big{\rvert} \bigg) < \infty ,
    \end{equation*}
    then
    \[\asdim{X} \le d.\]
    \end{lem}
    
    \begin{proof}
     For each ${i\in \{0,\hdots,d\}}$ define the set of all equivalence classes
    \begin{equation*}
        \Big\{ V^{(i),n} \colon n \in N_i  \Big\} \coloneq \Big\{ [ x ]_{i} \colon  x \in U^{\prn{i}} \Big\}  ,
    \end{equation*}
    where $N_i$ is an enumeration of the classes ${[\,\cdot\,]_{i}}$.
    Moreover, define the collections
    \begin{equation*}
        \Vcal^{(i)} \coloneq  \Big\{ V^{(i),n} \colon n \in N_i  \Big\}
    \end{equation*}
    and set
    \begin{equation*}
        \Vcal \coloneq  \Vcal^{(0)} \sqcup \hdots \sqcup  \Vcal^{(d)}.
    \end{equation*}
    We will show that $\Vcal$ is a covering that bounds the asymptotic dimension. 

    {From \eqref{clm:covering}} we have that ${\mathcal{C}}$ covers $X$, 
  therefore $\Vcal$ covers $X$.

    Let ${x, y \in V^{(i),n}}$ for some ${i\in \{0,\hdots,d\}}$ and ${n \in N_i}$. 
    Then ${x \sim_i y}$ and, since ${S < \infty }$, we have
    ${\distX{x}{y} \le S r}$.
    Thus, 
    \[\diam{V^{(i),n}} \le S r,\]
    for each ${i\in \{0,\hdots,d\}}$ and ${n \in N_i}$.
    In other words, ${\Vcal}$ is uniformly bounded.
      
    To show that $\Vcal$ is decomposed into $r$-separated collections, 
    we fix ${i\in \{0,\hdots,d\}}$ and ${n, n' \in N_i}$ with ${n \ne n'}$. 
    Suppose that ${x, y \in X}$, satisfying 
    \[(x,y) \in  \big(V^{(i),n}  \times  V^{(i),n'}  \big) \]
    and
    \[ {\distX{x}{y} \le r}. \]
    Note that ${\distX{x}{y} \le r}$
    implies
    ${[ x ]_{i} = [ y ]_{i}}$.
    This is a contradiction because ${n \ne n'}$. Therefore,
    \[ \Vcal^{(i)} =  \Big\{ V^{(i),n} \colon n \in N_i  \Big\} \]
    is an $r$-separated collection.

     Combining the above we obtain 
        ${\asdim{X} \le d}$.  \qedhere    
    \end{proof}

The main difficulty now is to control the overlaps of (the $r$-neighbourhoods of) the sets $U^{(i),j}_k$.
This is a place where our setup 
differs from that of Li, Liao and Winter \cite{LiLiaoWinter2023diagDim}. 
In fact, note that in \eqref{def:USets} one could in principle have some $x\in X$ such that both 
\[\norm{\unMatImg{i}{j}{k,k}\indF_x \unMatImg{i}{j}{k,k} } > \eta^2 \]
and 
\[\norm{\unMatImg{i}{j'}{k',k'}\indF_x \unMatImg{i}{j'}{k',k'} } > \eta^2,\]
and yet the following 
\[
    \big(\unMatImg{i}{j}{k,k}\indF_x \unMatImg{i}{j}{k,k} \big)
    \big(\unMatImg{i}{j'}{k',k'}\indF_x \unMatImg{i}{j'}{k',k'}
    \big) =0
\]
holds.
Here is exactly where Condition~\eqref{eq:condition6GenDiagDim} 
should play an important role, 
with the rough idea being that 
this condition should allow 
us to rule out the possibility that 
\[U^{(i),j}_k\cap U^{(i),j'}_{k'}\neq \emptyset\]
when $j\neq j'$.

An argument explaining a strategy to prove the above was detailed in the first version of this paper \cite{Kitsios2026GenDiagDimArXivV1},
but an anonymous referee found a substantial gap in the argument. 
Namely, at page 24 of \cite{Kitsios2026GenDiagDimArXivV1}
it was claimed that an inequality of the form 
\begin{equation}\label{ineq:Wrong}
    \norm{ \widehat{\phi}\Big(e_{k,k}^{(i),j} \Big)  \widetilde{T}_{z} 
\unMatImg{i}{j}{k,k}}  >  \delta^2/2.
\end{equation}
should hold with a certain fixed ${\delta>0}$. This is unfortunately not true, and the anonymous referee kindly gave us a concrete example where it fails.
For this example, 
let 
$X = \{x\}$, 
\[A = D_A = \mathcal B(\mathcal{H}),\] 
\[B =\ell^{\infty}(\mathbb N, \mathcal B(\mathcal H))\]
and 
\[F^{(0)} = F^{(1)} =\C.\]
Then, the completely positive contractive maps $\psi$ and $\phi$,
given by
\begin{align*}
    \psi(a) &= ((a, a, . . .), (\lambda a, a, a, . . .) ),
\end{align*}
and
\begin{align*}
    \phi(b, c) &= (1 - \lambda)b_1 + c_1
\end{align*}
for $a\in A$, $b,c\in B$ and an arbitrarily small $\lambda>0$, show that Inequality~\eqref{ineq:Wrong} is wrong.

This example does not negate Conjecture~\ref{conj:lowerBoundFPAlg}, but it does show that the tentative proof we gave requires a substantial modification.

\medskip
\subsection{Concluding remarks}

\subsubsection{Generalised diagonal dimension of Roe algebras}

We expect that one might be able to extract the asymptotic dimension of a uniformly locally finite metric space $X$ from its Roe algebra. 
We expect that,
by using similar techniques as in 
{Theorem~\ref{thm:upperBoundFPAlg} 
and 
Conjecture~\ref{conj:lowerBoundFPAlg},}
one might be able to prove
\begin{equation*}
         \ddimn{\CroeCartan{X}}{\Croe{\ell^2(X, \Hcal)}}{\ell^\infty(\N, \kop{\Hcal})} = \asdim{X}.
\end{equation*}
One technical obstruction that appears
is that in our {tentative} proof of Conjecture~\ref{conj:lowerBoundFPAlg}
we used that ${\ell^\infty(\N, \bop{\Hcal})}$ is closed in the strong operator topology,
whereas ${\ell^\infty(\N, \kop{\Hcal})}$ is not.
To overcome this, one could use that
the \cstar{} of finite\myhyp{}propagation operators on $X$
is the multiplier algebra of the Roe algebra ${\Croe{X}}$.
By overcoming this obstruction, one would be able to restate all components of the proofs of {Theorem~\ref{thm:upperBoundFPAlg} 
and 
Conjecture~\ref{conj:lowerBoundFPAlg},} 
and prove the equality above.
However, we were unable to fill in the details in order to prove this result.

\smallskip
\subsubsection{Noncommutative Cartan pairs and generalised diagonal dimension}

Li, Liao, and Winter showed that finite diagonal dimension of a nondegenerate \cstarSub{} implies that the subalgebra is a \cstarDiag{}.

\begin{thm}[{\cite[Theorem~2.10]{LiLiaoWinter2023diagDim}}] \label{thm:diagDimImpliesCdiagonal}
Let ${(D_A \subseteq A)}$ be a nondegenerate \cstarSub{} with ${\ddim{D_A}{A}<\infty}.$ 
Then $D_A$ is a \cstarDiag{} in $A$.
\end{thm}

An open question is whether a noncommutative version of the above holds. 
More precisely, given a \cstar{} $B$ and
a nondegenerate \cstarSub{} ${(D_A \subseteq A)}$ such that 
    \[\ddimn{D_A}{A}{B}<\infty,\]
does it follow that $D_A$ a noncommutative Cartan subalgebra in $A$?

It is clear that choosing a suitable \cstar{} $B$ as 
coefficients
plays a crucial role in the above question.
This brings us to another direction that can be explored.

\smallskip
\subsubsection{Coefficients in generalised diagonal dimension}

It would be of interest to determine appropriate coefficients $B$ in the definition of generalised diagonal dimension.
Note that, 
{when $X$ is countably infinite,} 
the \cstar{} $\ell^\infty(\N,\bop{\Hcal})$ {in Conjecture~\ref{intro:thm:diagDimFPAlg}} is (non\myhyp{}canonically) isomorphic to the \cstarSub{} $\CfpCartan{X}$.
It remains to be seen whether computing the generalised diagonal dimension of a pair ${(D_A \subseteq A)}$ with respect to ${D_A}$, that is, computing ${\ddimn{D_A}{A}{D_A}}$,
yields meaningful results.

\bibliographystyle{amsplain}
\bibliography{references.bib}

\providecommand{\bysame}{\leavevmode\hbox to3em{\hrulefill}\thinspace}
\providecommand{\MR}{\relax\ifhmode\unskip\space\fi MR }
\providecommand{\MRhref}[2]{%
  \href{http://www.ams.org/mathscinet-getitem?mr=#1}{#2}
}
\providecommand{\href}[2]{#2}
\begin{thebibliography}{10}

\bibitem{BaudierBragaFarahKhukhroVignatiWillett2022UniformRoeAlgRigid}
Florent~P. Baudier, Bruno~M. Braga, Ilijas Farah, Ana Khukhro, Alessandro Vignati, and Rufus Willett, \emph{Uniform {Roe} algebras of uniformly locally finite metric spaces are rigid}, {Inventiones Mathematicae} \textbf{230} (2022), no.~3, 1071\mydhyp{}1100.

\bibitem{BragaChungLi2020CoarseBaumConnesandRigidity}
Bruno~M. Braga, Yeong~Chyuan Chung, and Kang Li, \emph{{Coarse Baum\mydhyp{}Connes conjecture and rigidity for Roe algebras}}, {Journal of Functional Analysis} \textbf{279} (2020), no.~9, 20.

\bibitem{BragaFarahVignati2020EmbeddingsUniformRoe}
Bruno~M. Braga, Ilijas Farah, and Alessandro Vignati, \emph{Embeddings of uniform {Roe} algebras}, {Communications in Mathematical Physics} \textbf{377} (2020), no.~3, 1853\mydhyp{}1882.

\bibitem{BragaFarahVignati2021UniformRoeCoronas}
\bysame, \emph{Uniform {Roe} coronas}, {Advances in Mathematics} \textbf{389} (2021), 35.

\bibitem{BragaFarahVignati2022GeneralUniformRoeRigidity}
\bysame, \emph{General uniform {Roe} algebra rigidity}, {Annales de l'Institut Fourier} \textbf{72} (2022), no.~1, 301\mydhyp{}337.

\bibitem{BragaVignati2023GeldandTypeDuality}
Bruno~M. Braga and Alessandro Vignati, \emph{A {Gelfand}\myhyp{}type duality for coarse metric spaces with property {A}}, {IMRN. International Mathematics Research Notices} \textbf{2023} (2023), no.~11, 9799\mydhyp{}9843.

\bibitem{Exel2011noncommCartan}
Ruy Exel, \emph{Noncommutative {{Cartan}} subalgebras of {{C}}$^\ast$\myhyp{}algebras}, {The} {New} {York} {Journal} of {Mathematics} \textbf{17} (2011), 331\mydhyp{}382.

\bibitem{Gromov1993AsymptoticInvariantsInfiniteGroups}
Mikhael Gromov, \emph{Geometric group theory. {Volume} 2: {Asymptotic} invariants of infinite groups}, {London Mathematical Society Lecture Note Series}, vol. 182, {Cambridge: Cambridge University Press}, 1993.

\bibitem{GuentnerWillettYu2017dynAsDim}
Erik Guentner, Rufus Willett, and Guoliang Yu, \emph{Dynamic asymptotic dimension: relation to dynamics, topology, coarse geometry, and {{C}}$^{\ast}$\myhyp{}algebras}, {Mathematische Annalen} \textbf{367} (2017), no.~1\mydhyp{}2, 785\mydhyp{}829.

\bibitem{HigsonRoe2000AnalyticKHomology}
Nigel Higson and John Roe, \emph{Analytic {K}\myhyp{}homology}, {Oxford Mathematical Monographs}, {Oxford: Oxford University Press}, 2000.

\bibitem{HigsonRoeYu1993CoarseMayerVietoris}
Nigel Higson, John Roe, and Guoliang Yu, \emph{{A coarse Mayer\mydhyp{}Vietoris principle}}, {Mathematical Proceedings of the Cambridge Philosophical Society} \textbf{114} (1993), no.~1, 85\mydhyp{}97.

\bibitem{KhukhroLiVigoloZhang2021StructureAsymptoticExpanders}
Ana Khukhro, Kang Li, Federico Vigolo, and Jiawen Zhang, \emph{On the structure of asymptotic expanders}, {Advances in Mathematics} \textbf{393} (2021), 35, Id/No. 108073.

\bibitem{Kitsios2026GenDiagDimArXivV1}
Christos Kitsios, \emph{{Generalised diagonal dimension and applications to large-scale geometry}}, 2026, {Preprint, arXiv:2604.07237v1}.

\bibitem{Kitsios2025PhDThesis}
\bysame, \emph{{Generalised Diagonal Dimension and Roe-Like Algebras}}, Ph.D. thesis, Georg\myhyp{}August\myhyp{}Universit{\"{a}}t G{\"{o}}ttingen, 2026.

\bibitem{Kumjian1986onCDiagonals}
Alexander Kumjian, \emph{On {C$^\ast$}\myhyp{}diagonals}, {Canadian Journal of Mathematics} \textbf{38} (1986), 969\mydhyp{}1008.

\bibitem{KwasniewskiMeyer2020NoncommutativeCartan}
Bartosz~Kosma Kwa{\'{s}}niewski and Ralf Meyer, \emph{Noncommutative {Cartan} {C$^\ast$}\myhyp{}subalgebras}, {Transactions of the American Mathematical Society} \textbf{373} (2020), no.~12, 8697\mydhyp{}8724.

\bibitem{LiLiaoWinter2023diagDim}
Kang Li, Hung\myhyp{}Chang Liao, and Wilhelm Winter, \emph{The diagonal dimension of sub\myhyp{{C}}$^{\ast}$\myhyp{}algebras}, 2023, {Preprint, arXiv:2303.16762}.

\bibitem{LiSpakulaZhang2023MeasuredAsymptoticExpandersRigidity}
Kang Li, J{\'{a}n}~{\v{S}}pakula, and Jiawen Zhang, \emph{Measured asymptotic expanders and rigidity for {Roe} algebras}, {IMRN. International Mathematics Research Notices} \textbf{2023} (2023), no.~17, 15102\mydhyp{}15154.

\bibitem{MartinezVigolo2023RoeViaMods}
Diego Mart\'{i}nez and Federico Vigolo, \emph{{{Roe}} algebras of coarse spaces via coarse geometric modules}, 2023, {Preprint, arXiv:2312.08907}.

\bibitem{MartinezVigolo2025RigidityBoundedGeometry}
\bysame, \emph{{{C}}$^{\ast}$\myhyp{}rigidity of bounded geometry metric spaces}, {Publications Math{\'{e}}matiques} \textbf{141} (2025), 333\mydhyp{}348.

\bibitem{NowakYu2023LargeScaleGeometry}
Piotr~W. Nowak and Guoliang Yu, \emph{Large scale geometry}, 2nd edition ed., {EMS Textbooks in Mathematics}, {Berlin: European Mathematical Society (EMS)}, 2023.

\bibitem{Renault2008cartanSub}
Jean Renault, \emph{Cartan subalgebras in {C$^*$\myhyp{}algebras}}, {Irish Mathematical Society Bulletin} \textbf{61} (2008), 29\mydhyp{}63.

\bibitem{Roe1988IndexThm1}
John Roe, \emph{An index theorem on open manifolds {I}}, {Journal of Differential Geometry} \textbf{27} (1988), no.~1, 87\mydhyp{}113.

\bibitem{Roe1988IndexThm2}
\bysame, \emph{An index theorem on open manifolds {II}}, {Journal of Differential Geometry} \textbf{27} (1988), no.~1, 115\mydhyp{}136.

\bibitem{Roe1993CoarseCohomologyIndexTheory}
\bysame, \emph{Coarse cohomology and index theory on complete {Riemannian} manifolds}, {Memoirs of the American Mathematical Society}, vol. 497, {Providence, RI: American Mathematical Society (AMS)}, 1993.

\bibitem{Roe1996IndexTheoryCoarseGeometryTopology}
\bysame, \emph{Index theory, coarse geometry, and topology on manifolds}, {Regional Conference Series in Mathematics}, vol.~90, {Providence, RI: American Mathematical Society (AMS)}, 1996.

\bibitem{Roe2003lecturesCrsGeometry}
\bysame, \emph{Lectures on coarse geometry}, {University Lecture Series}, vol.~31, {Providence, RI: American Mathematical Society (AMS)}, 2003.

\bibitem{Sako2014ProperyAandONLPropery}
Hiroki Sako, \emph{Property {A} and the operator norm localization property for discrete metric spaces}, {Journal f{\"{u}}r die Reine und Angewandte Mathematik} \textbf{690} (2014), 207\mydhyp{}216.

\bibitem{SpakulaWillett2013RigidityRoeAlgebras}
J{\'{a}}n {\v{S}}pakula and Rufus Willett, \emph{{On rigidity of Roe algebras}}, {Advances in Mathematics} \textbf{249} (2013), 289\mydhyp{}310.

\bibitem{Vignati2026RigidityUniformRoe}
Alessandro Vignati, \emph{The rigidity problem for uniform {Roe} algebras}, 2026, {Preprint, arXiv:2602.01998}.

\bibitem{WhiteWillett2020CartanUniformRoe}
Stuart White and Rufus Willett, \emph{Cartan subalgebras in uniform {Roe} algebras}, {Groups, Geometry, and Dynamics} \textbf{14} (2020), no.~3, 949\mydhyp{}989.

\bibitem{Willett2009NotesPropertyA}
Rufus Willett, \emph{Some notes on property {A}}, {Limits of graphs in group theory and computer science}, {Lausanne: EPFL Press/distr. by CRC Press}, 2009, p.~191\mydhyp{}284.

\bibitem{WillettYu2020HigherIndexTheory}
Rufus Willett and Guoliang Yu, \emph{Higher index theory}, {Cambridge Studies in Advanced Mathematics}, vol. 189, {Cambridge: Cambridge University Press}, 2020.

\bibitem{WinterZacharias2009orderZero}
Wilhelm Winter and Joachim Zacharias, \emph{Completely positive maps of order zero}, {M{\"{u}}nster} {Journal} of {Mathematics} \textbf{2} (2009), no.~1, 311\mydhyp{}324.

\bibitem{WinterZacharias2010nucDim}
\bysame, \emph{The nuclear dimension of {{C}}$^{\ast}$\myhyp{}algebras}, {Advances} in {Mathematics} \textbf{224} (2010), no.~2, 461\mydhyp{}498.

\bibitem{Yu1997ZeroInTheSpectrumAsDim}
Guoliang Yu, \emph{Zero-in-the-spectrum conjecture, positive scalar curvature and asymptotic dimension}, {Inventiones Mathematicae} \textbf{127} (1997), no.~1, 99\mydhyp{}126.

\bibitem{Yu1998NovikovAsDim}
Guoliang Yu, \emph{The {Novikov} conjecture for groups with finite asymptotic dimension}, {Annals of Mathematics. Second Series} \textbf{147} (1998), no.~2, 325\mydhyp{}355.

\end{thebibliography}

\end{document}